\documentclass[10pt]{amsart}
\title[]{Invariant distributions on projective spaces over local fields}
\author[]{Guyan Robertson}
\address{School of Mathematics and Statistics, University of Newcastle, NE1 7RU, England, U.K.}
\email{a.g.robertson@ncl.ac.uk}
\date{July 1, 2010}
\subjclass{Primary 20F65, 20G25, 51E24.}
\keywords{Buildings, boundary distributions}

\usepackage{amsmath}
\usepackage{amscd}
\usepackage{amssymb}

\input{prepictex}
\input{pictex}
\input{postpictex}
\chardef\bslash=`\\ % p. 424, TeXbook

\makeatletter
\def\verbatim{\interlinepenalty\@M \@verbatim
  \leftskip\@totalleftmargin\advance\leftskip2pc
  \frenchspacing\@vobeyspaces \@xverbatim}
\makeatother
\newtheorem{theorem}{Theorem}[section]
\newtheorem{corollary}[theorem]{Corollary}
\newtheorem{lemma}[theorem]{Lemma}

\theoremstyle{definition}

\newtheorem{definition}[theorem]{Definition}

\newcommand{\cl}[1]{{\mathcal{#1}}}
\newcommand{\bb}[1]{{\mathbb{#1}}}
\newcommand{\fk}[1]{{\mathfrak{#1}}}

\newcounter{picture}

\newcommand{\e}{{\varepsilon}}

\newcommand{\PGL}{{\text{\rm{PGL}}}}

\newcommand{\GL}{{\text{\rm{GL}}}}

\newcommand{\vol}{{\text{\rm{vol}}}}
\newcommand{\id}{{\bf 1}}

\begin{document}

\begin{abstract}
Let $\Gamma$ be an $\widetilde A_n$ subgroup of $\PGL_{n+1}(\bb K)$, with $n\ge 2$, where $\bb K$ is a local field with residue field of order $q$ and let $\bb P^n_{\bb K}$ be projective $n$-space over $\bb K$.
The module of coinvariants $H_0(\Gamma; C(\bb P^n_{\bb K},\bb Z))$ is
shown to be finite.
Consequently there is no nonzero $\Gamma$-invariant $\bb Z$-valued distribution on $\bb P^n_{\bb K}$.
\end{abstract}

\maketitle

\section{Introduction}\label{intro}

Let $\bb K$ be a nonarchimedean local field with residue field $k$ of order $q$ and uniformizer $\pi$.
Denote by $\bb P^n_{\bb K}$ the set of one dimensional subspaces of the vector space $\bb K^{n+1}$, i.e. the set of points in projective $n$-space over $\bb K$. Then
$\bb P^n_{\bb K}$ is a compact totally disconnected space with the quotient topology inherited from $\bb K^{n+1}$, and there is a continuous action of $G=\PGL_{n+1}(\bb K)$ on $\bb P^n_{\bb K}$.

Let $\Gamma$ be a lattice subgroup of $G$. The abelian group $C(\bb P^n_{\bb K},\bb Z)$ of continuous integer-valued functions on $\bb P^n_{\bb K}$ has the structure of a $\Gamma$-module and the module of coinvariants $C(\bb P^n_{\bb K},\bb Z)_{\Gamma}= H_0(\Gamma; C(\bb P^n_{\bb K},\bb Z))$ is a finitely generated group. Now suppose that $\Gamma$ is an $\widetilde A_n$ group \cite{ca1,cs}, i.e. $\Gamma$ acts freely and transitively on the vertex set of the Bruhat-Tits building of $G$, which has type $\widetilde A_n$.
A free group is an $\widetilde A_1$ group since it acts freely and transitively
on the vertex set of a tree, which is a building of type $\widetilde A_1$. For $n\ge 2$,  the $\widetilde A_n$ groups are unlike free groups. This article proves the following.

 \begin{theorem}\label{finite}
 If $\Gamma$ is an $\widetilde A_n$ subgroup of $\PGL_{n+1}(\bb K)$, where $n\ge 2$, then $C(\bb P^n_{\bb K},\bb Z)_{\Gamma}$ is a finite group.
 \end{theorem}

The proof depends upon the fact that $\Gamma$ has Kazhdan's property (T). A {\it distribution} on $\bb P^n_{\bb K}$ is a finitely additive $\bb Z$-valued measure $\mu$ defined on the  clopen subsets of $\bb P^n_{\bb K}$.
\begin{corollary}\label{cor}
If $\Gamma$ is an $\widetilde A_n$ subgroup of $\PGL_{n+1}(\bb K)$, where $n\ge 2$, then there is no nonzero $\Gamma$-invariant $\bb Z$-valued distribution on $\bb P^n_{\bb K}$.
\end{corollary}
This contrasts strongly with the main result of \cite{rob} concerning boundary distributions associated with
finite graphs.
A torsion free lattice subgroup $\Gamma$ of $\PGL_2(\bb K)$ is a free group, of rank $r$ say.
It was shown in \cite{rob} that in this case the group of $\Gamma$-invariant $\bb Z$-valued distributions on $\bb P^1_{\bb K}$ is isomorphic to $\bb Z^r$. In particular, there are many such distributions.

\section{Background}
\subsection{The Bruhat-Tits building}\label{building}

If $\bb K$ is a local field, with discrete valuation $v: \bb K^\times\to\bb Z$, let
$\cl O=\{x\in \bb K:v(x)\ge0\}$ and let $\pi\in \bb K$ satisfy $v(\pi)=1$.
A \textit{lattice} $L$ is an  $\cl O$-submodule of $\bb K^{n+1}$ of rank $n+1$.
In other words $L=\cl O e_1 + \cl O e_2 + \dots + \cl O e_{n+1}$, for some basis
$\{e_1, e_2, \dots ,e_{n+1}\}$ of $\bb K^{n+1}$.
Two lattices $L_1$ and $L_2$ are \textit{equivalent} if $L_1=\alpha L_2$ for some $\alpha\in \bb K^\times$.
The Bruhat-Tits building of $\PGL_{n+1}(\bb K)$ is a two dimensional simplicial complex $\Delta$ whose vertices are equivalence classes of lattices in $\bb K^{n+1}$ \cite{ron}.
Two lattice classes $[L_0], [L_1]$ are \textit{adjacent} if, for suitable representatives
 $L_1, L_2$, we have $L_0 \subset L_1 \subset \pi^{-1} L_0$. A \textit{simplex} is a set of pairwise adjacent lattice classes. The maximal simplices (\textit{chambers}) are  the sets $\{[L_0], [L_1], \dots ,[L_n]\}$ where
$L_0 \subset L_1 \subset \dots \subset L_n \subset \pi^{-1} L_0$. These inclusions determine a canonical ordering of the vertices in a chamber, up to cyclic permutation.
Each vertex $v$ of $\Delta$ has a {\it type} $\tau(v) \in \bb Z/(n+1)\bb Z$, and each chamber of $\Delta$ has exactly one vertex of each type.
If the Haar measure on $\bb K^{n+1}$ is normalized so that $\cl O^{n+1}$ has measure $1$ then the type map may be defined by
$\tau([L])=\log_q(\vol(L)) + (n+1)\bb Z$. The cyclic ordering of the vertices of a chamber coincides with the natural ordering given by the vertex types (Figure \ref{A3chamber}).
Let $E^1$ denote the set of directed edges $e=(x,y)$ of $\Delta$ such that $\tau(y)=\tau(x)+1$.
Write $o(e)=x$ and $t(e)=y$.
The subgraph of the 1-skeleton of $\Delta$ with edge set $E^1$ is studied in \cite{csz,lsv}.
\refstepcounter{picture}
\begin{figure}[htbp]
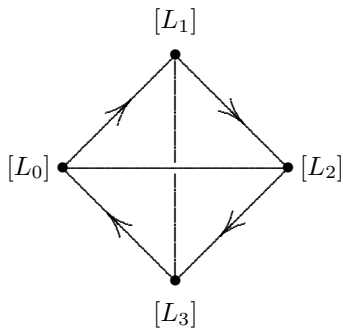
\label{A3chamber}
\centerline{
\beginpicture
\setcoordinatesystem units <1.5cm, 1.5cm>
\setplotarea x from -2 to 2, y from  -1.2 to  1.2
\put {$\bullet$} at 1 0
\put {$\bullet$} at 0 1
\put {$\bullet$} at -1 0
\put {$\bullet$} at 0 -1
\arrow <10pt> [.2, .67] from -0.5 0.5 to -0.4 0.6
\arrow <10pt> [.2, .67] from  0.5 0.5 to 0.6 0.4
\arrow <10pt> [.2, .67] from  0.5 -0.5 to 0.4 -0.6
\arrow <10pt> [.2, .67] from  -0.5 -0.5 to -0.6 -0.4
\put {$[L_0]$} at -1.3 0
\put {$[L_1]$} at 0 1.3
\put {$[L_2]$} at 1.3 0
\put {$[L_3]$} at 0 -1.3
\setlinear \plot 1 0  -1 0  0 1  1 0  0 -1  -1 0 /
\plot 0 0.05   0 1 /
\plot 0 -0.05   0 -1 /
\endpicture
}
\caption{$\widetilde A_3$ case: cyclic ordering of the vertices of a chamber}
\end{figure}
\begin{lemma}\label{parallel edges}
Let $C$ be a chamber of $\Delta$. Then $C$ contains $n+1$ directed edges $e\in E^1$.
\end{lemma}

\begin{proof}
By \cite[Chapter 9.2]{ron}, there is a basis $(e_1,\ldots,e_{n+1})$ of $\bb K^{n+1}$ such that the vertices of $C$ are the classes of the lattices
\begin{align*}
L_0&=\pi\cl O e_1+\pi\cl O e_2+\pi\cl O e_3+\cdots+\pi\cl O e _{n+1} \\
L_1&=\cl O e_1+\pi\cl O e_2+\pi\cl O e_3+\cdots+\pi\cl O e _{n+1} \\
L_2&=\cl O e_1+\cl O e_2+\pi\cl O e_3+\cdots+\pi\cl O e _{n+1} \\
\dots & \dots\\
L_n&=\cl O e_1+\cl O e_2+\cl O e_3+\cdots+\pi\cl O e _{n+1}.
\end{align*}
Define $L_{n+1}=L_0$. Then the edges $C$ which lie in $E^1$ are $([L_k],[L_{k+1}])$, where $0\le k\le n$.
\end{proof}

The building $\Delta$ is of type $\widetilde A_n$ and the action of $\GL_{n+1}(\bb K)$  on the set of lattices induces an action of $\PGL_{n+1}(\bb K )$ on $\Delta$ which is transitive on the vertex set. The action of $\PGL_{n+1}(\bb K )$ on $\Delta$ is \textit{type rotating} in the sense that, for each $g\in\PGL_{n+1}(\bb K )$, there exists $i \in \bb Z/(n+1)\bb Z$ such that $\tau(gv) = \tau(v)+i$ for all vertices $v \in \Delta$.

Fix a vertex~$v_0\in \Delta$ of type $0$, and let $\Pi(v_0)$ be
the set of vertices adjacent to $v_0$. Then $\Pi(v_0)$ has a natural incidence
structure: if $u,v\in\Pi(v_0)$ are distinct, then $u$ and $v$ are
{\it incident\/} if $u$, $v$ and $v_0$ lie in a common chamber of~$\Delta$.
If $v_0$ is the lattice class $[L_0]$, then $\Pi(v_0)$ consists of the classes~$[L]$ where $L_0\subset L\subset \pi^{-1} L_0$,
and one can associate to~$[L]\in\Pi(v_0)$ the subspace $v=L/L_0$ of
$\pi^{-1}L_0/L_0\cong k^{n+1}$.
Thus we may identify $\Pi(v_0)$ with the flag complex of
 subspaces of the vector space $k^{n+1}$.
Under this identification, a vertex $v\in\Pi(v_0)$ has type $\tau(v)=\dim(v)+\bb Z/(n+1)\bb Z$ where $\dim(v)$ is the
dimension of~$v$ over~$k$.
A chamber $C$ of $\Delta$ which contains $v_0$ has vertices $v_0, v_1, \dots , v_n$ where
$(0)=v_0\subset v_1\subset\dots\subset v_n\subset k^{n+1}$ is a complete flag.
For brevity, write $C=\{v_0\subset v_1\subset\dots\subset v_n\}$.

\begin{definition}
  If $e=([L_0],[L_1])\in E^1$, where $L_0\subset L_1\subset \pi^{-1} L_0$ and $\tau([L_1])=\tau([L_0])+1$, then define $\Omega(e)$ to be the set of lines $\ell\in \bb P^n_{\bb K}$ such that $L_1=L_0 + (\ell \cap \pi^{-1}L_0)$.
  The sets $\Omega(e)$, $e\in E^1$, form a basis for the topology on $P^n_{\bb K}$ (c.f. \cite[Ch.II.1.1]{ser}, \cite[1.6]{ads}).
\end{definition}

\begin{lemma}\label{PBL}
If $e\in E^1$,
then $\Omega(e)$ may be expressed as a disjoint union of $q^n$ sets
\begin{equation}\label{PB}
\Omega(e)=\bigsqcup_{\substack{o(e')=t(e)\\ \Omega(e') \subset \Omega(e)} }\Omega(e')\,.
\end{equation}
\end{lemma}

\begin{proof}
  Let $e=([L_0],[L_1])\in E^1$, where $L_0\subset L_1\subset \pi^{-1} L_0$ and $\tau([L_1])=\tau([L_0])+1$.
  If $\ell\in \Omega(e)$ then $L_1=L_0 + (\ell \cap \pi^{-1}L_0)$.
  Choose $e'=([L_1],[L_2])$ where $L_2=L_0 + (\ell \cap \pi^{-2}L_0)$.
Now $L_0\subset L_1\subset L_2\subset \pi^{-1}L_1$ and $L_2/L_1$ is a 1-dimensional subspace of $\pi^{-1}L_1/L_1\cong k^{n+1}$.
  Moreover, $L_2/L_1$ is not incident with the $n$-dimensional subspace $\pi^{-1}L_0/L_1$ of $\pi^{-1}L_1/L_1\cong k^{n+1}$. There are precisely $q^n$ such 1-dimensional subspaces of $k^{n+1}$, each of which corresponds to an edge $e'\in E^1$.
\end{proof}

\begin{lemma} If $\xi$ is a fixed vertex of $\Delta$, then $\bb P^n_{\bb K}$ may be expressed as a disjoint union
\begin{equation}\label{PA}
\bb P^n_{\bb K}=\bigsqcup_{o(e)=\xi}\Omega(e)\,.
\end{equation}
\end{lemma}

\begin{proof}
Let $\xi=[L_0]$, where $L_0$ is a lattice. If $\ell \in \bb P^n_{\bb K}$, define the lattice $L_1= L_0 +(\ell \cap \pi^{-1}L_0)$.
Then $L_0\subset L_1\subset\pi^{-1} L_0$ and
$\tau([L_1])=\tau([L_0])+1$, since $L_0$ is maximal in $L_1$.
Thus the edge $e=([L_0],[L_1])$ lies in $E^1$, and $\ell\in \Omega(e)$.
\end{proof}

\begin{lemma}
  Let $C$ be a chamber of $\Delta$ and denote the directed edges of $C\cap E^1$ by $e_0, e_1, \dots , e_n$.
  Then $\bb P^n_{\bb K}$ may be expressed as a disjoint union
\begin{equation}\label{PC}
\bb P^n_{\bb K}=\bigsqcup_{i=0}^n\Omega(e_i)\,.
\end{equation}

\begin{proof}
Let $C$ have vertex set $\{[L_0], [L_1], \dots ,[L_n]\}$ where
$L_0 \subset L_1 \subset \dots \subset L_n \subset \pi^{-1} L_0$.
Let $\ell=\bb K a \in \bb P^n_{\bb K}$, where $a\in \bb K^{n+1}$ is scaled so that
$a\in \pi^{-1}L_0-L_0$. Then $a\in L_{i+1}-L_i$ for some $i$, where $L_{i+1}/L_i\cong k$
and $L_{n+1}=\pi^{-1}L_0$.
Thus $\ell\in \Omega(e_i)$.
\end{proof}

\end{lemma}

\subsection{$\widetilde A_n$ groups}\label{Antildegroups}

From now on let $\Pi=\Pi(v_0)$, the set of neighbours of the fixed vertex~$v_0\in \Delta$. Thus $\Pi$ is isomorphic to the flag complex of subspaces of $k^{n+1}$ and a chamber $C$ of $\Delta$ which contains $v_0$ is a complete flag $\{v_0\subset v_1\subset\dots\subset v_n\}$.
For $1\le r \le n$, let $\Pi_r=\{u\in\Pi(v_0):\dim u=r\}$.

Now suppose that $\Gamma$ is an $\widetilde A_n$ group i.e. $\Gamma$ acts freely and transitively on the vertex set of $\Delta$ \cite{ca1,cs}.
Then for each $v\in \Pi(v_0)$, there is a unique element $g_v \in \Gamma$ such that $g_vv_0 = v$. If $v\in \Pi(v_0)$, then $g^{-1}_v v_0$ also lies
in $\Pi(v_0)$, and $\lambda(v)=g_v^{-1}v_0$ defines an involution
$\lambda :\Pi(v_0)\to \Pi(v_0)$ such that $g_{\lambda(v)}=g_v^{-1}$.
Let $\cl T = \{(u,v,w)\in \Pi(v_0)^3 : g_ug_vg_w = 1\}$.
 If $(u,v,w)\in\cl T$ then $w$ is uniquely determined by $(u,v)$ and there is a bijective correspondence
 between triples $(u,v,w)\in\cl T$ and directed triangles $(v_0,\lambda(u),v)$ of $\Delta$ containing $v_0$.
 By \cite[Proposition 2.2]{cmsz}, the abstract group $\Gamma$
 has a presentation with generating set $\{g_v : v \in\Pi(v_0)\}$ and relations
\begin{subequations}\label{Gamma}
\begin{eqnarray}
g_ug_{\lambda(u)} &=& 1, \quad\, u\in \Pi(v_0); \label{ga}\\
g_ug_vg_w &=& 1, \quad (u,v,w)\in \cl T. \label{gb}
\end{eqnarray}
\end{subequations}

If $u\in \Pi(v_0)$ and then $\tau(g_uv_0)=\tau(u)=\tau(u)+\tau(v_0)$.
Hence $\tau(g_ux)=\tau(u)+\tau(x)$ for each vertex $x$ of $\Delta$, since $g_u$ is type rotating.
In particular, if $u,v\in \Pi(v_0)$ then
\begin{equation}\label{one}
\tau(g_ug_vv_0)=\tau(u)+\tau(v).
\end{equation}
\noindent
It follows from (\ref{one}) that
$$\tau(\lambda(u))=-\tau(u)$$
 for each $u\in \Pi$. Also, if
 $(u,v,w)\in \cl T$, then
$$\tau(u)+\tau(v)+\tau(w) = 0.$$

Let $C=\{v_0\subset v_1\subset\dots\subset v_n\}$ be a chamber of $\Delta$ containing $v_0$.
Since the vertices $v_{i-1}$ and $v_i$ are adjacent, so are the vertices $v_0=g_{v_{i-1}}^{-1}v_{i-1}$ and
$g_{v_{i-1}}^{-1}g_{v_i}v_0=g_{v_{i-1}}^{-1}v_i$.
Also $\tau(g_{v_{i-1}}^{-1}g_{v_i}v_0)=\tau(v_i)-\tau(v_{i-1})=1$. Therefore $g_{v_{i-1}}^{-1}g_{v_i}=g_{a_i}$ where $a_i\in\Pi_1$, $v_{n+1}=v_0$ and $g_{v_0}=1$.
Thus $g_{a_1}g_{a_2}\dots g_{a_k}=g_{v_k}$ ($1\le k\le n$) and  $g_{a_1}g_{a_2}\dots g_{a_{n+1}}=1$.

The $(n+1)$-tuple $\sigma(C)=(a_1,a_2,\dots ,a_{n+1})\in \Pi_1^{n+1}$ is uniquely determined by
the chamber $C$ containing $v_0$.
Denote by $\fk S$ the set of all $(n+1)$-tuples $\sigma(C)$ associated with such chambers $C$.
If $u\in \Pi(v_0)$ with $\dim(u)=k$, then $u$ is a vertex of
a chamber $C$ containing $v_0$. Therefore
\begin{equation}\label{decomposition}
g_u=g_{a_1}g_{a_2}\dots g_{a_k}, \quad \text{where}\quad a_i\in \Pi_1, 1\le i \le k.
\end{equation}
 In particular, the set $\{g_a : a\in\Pi_1\}$ generates $\Gamma$. Since $g_{\lambda(u)}=g_u^{-1}$,
 we have
 \begin{equation}\label{inverse}
 g_{\lambda(u)}=g_{a_{i+1}}\dots g_{a_{n+1}}.
\end{equation}
 Note that the expression (\ref{decomposition}) for $g_u$ is not unique,
but depends on the choice of the chamber $C$ containing $u$ and $v_0$.
An edge in $E^1$ has the form $(x, g_ax)$ where $a\in \Pi_1$.

\begin{lemma}\label{presentation}
The $\widetilde A_n$ group $\Gamma$ has a presentation with generating set $\{g_a: a\in \Pi_1\}$ and relations
\begin{equation}\label{new relations}
g_{a_1}g_{a_2}\dots g_{a_{n+1}}=1, \qquad (a_1,a_2,\dots ,a_{n+1})\in \fk S.
\end{equation}
\end{lemma}

\begin{proof}
It is enough to show that the relations (\ref{Gamma}) follow from the relations (\ref{new relations}).
Let $(u,v,w)\in \cl T$ with $\dim(u)=i, \dim v=j$ and  $\dim w=k$,
where $i+j+k \equiv 0 \mod (n+1)$. Choose a chamber $C=\{v_0\subset v_1\subset\dots\subset v_n\}$ containing $\{v_0, g_uv_0, g_ug_vv_0\}$.
Let $(a_1,a_2,\dots ,a_{n+1})=\sigma(C)\in \Pi_1^{n+1}$ be the element of $\fk S$ determined by $C$.
Then $g_uv_0$ is the vertex of $C$ of type $i$, so $g_u=g_{a_1}g_{a_2}\dots g_{a_i}$.

Suppose that $j<n+1-i$. Then $g_ug_vv_0$ is the vertex of $C$ of type $i+j$ and $g_ug_v=g_{a_1}g_{a_2}\dots g_{a_{i+j}}$.
Thus $g_v=g_{a_{i+1}}\dots g_{a_{i+j}}$ and $g_w=g_{a_{i+j+1}}\dots g_{a_{n+1}}$.
Therefore
$$g_ug_vg_w= g_{a_1}g_{a_2}\dots g_{a_{n+1}}.$$
Suppose that $j>n+1-i$. Then $g_ug_vv_0$ has type $i+j-n-1$ and
$$g_ug_v=g_{a_1}g_{a_2}\dots g_{a_{i+j-n-1}}=g_{a_1}g_{a_2}\dots g_{a_{n+1}}g_{a_1} \dots  g_{a_{i+j-n-1}}.$$
Thus $g_v=g_{a_{i+1}}\dots g_{a_{n+1}}g_{a_1} \dots  g_{a_{i+j-n-1}}$
 and $g_w=g_{a_{i+j-n}}\dots g_{a_{n+1}}$.
Therefore
$$g_ug_vg_w= (g_{a_1}g_{a_2}\dots g_{a_{n+1}})^2.$$
In each case the relations (\ref{gb}) follow from the relations (\ref{new relations}).
The same is true for the relations (\ref{ga}), by equation (\ref{inverse}).
\end{proof}

\section{The coinvariants}

If $\Gamma$ is an $\widetilde A_n$ group acting on $\Delta$, then $\Gamma$ acts on $\bb P^n_{\bb K}$, and the abelian group $C(\bb P^n_{\bb K},\bb Z)$ has the structure of a $\Gamma$-module, with
$(g\cdot f)(\ell)=f(g^{-1}\ell),\, g\in \Gamma,\, \ell\in\bb P^n_{\bb K}$.  The module of coinvariants, $C(\bb P^n_{\bb K},\bb Z)_\Gamma$, is the quotient of
$C(\bb P^n_{\bb K},\bb Z)$ by the submodule generated by
$\{g\cdot f-f : g\in\Gamma, f\in C(\bb P^n_{\bb K},\bb Z)\}$.
If $f\in C(\bb P^n_{\bb K},\bb Z)$ then let $[f]$ denote its class in $C(\bb P^n_{\bb K},\bb Z)_\Gamma$.
Also, let $\id$ denote the constant function defined by $\id(\ell)=1$ for $\ell\in\bb P^n_{\bb K}$, and let $\e=[\id]$.

If $e\in E^1$, let $\chi_e$ be the characteristic function of $\Omega(e)$.
For each $g\in\Gamma$, the functions $\chi_e$ and $g\cdot\chi_e=\chi_{ge}$ project to the same element in
$C(\bb P^n_{\bb K},\bb Z)_\Gamma$.
Any edge $e\in E^1$ is in the $\Gamma$-orbit of some edge $(v_0,g_av_0)$, where $a\in \Pi_1$ is uniquely determined by $e$.
Therefore it makes sense to denote by $[a]$ the class of $\chi_e$ in $C(\bb P^n_{\bb K},\bb Z)_\Gamma$.

\begin{lemma}\label{fg}
The group $C(\bb P^n_{\bb K},\bb Z)_\Gamma$ is finitely generated, with generating set $\{[a]:a\in\Pi_1\}$.
\end{lemma}

\begin{proof}
Every clopen set $V$ in $\bb P^n_{\bb K}$ may be expressed as a finite disjoint union of sets of the form $\Omega(e)$, $e\in E^1$. Any function $f\in C(\bb P^n_{\bb K},\bb Z)$ is bounded, by compactness of $\bb P^n_{\bb K}$, and so takes finitely many values $n_i\in\bb Z$. Therefore $f$ may be expressed as a finite sum $f=\sum_j n_j\chi_{e_j}$, with
$e_j\in E^1$. The result follows, since $\{[\chi_e] : e\in E^1\}=\{[a] : a\in \Pi_1\}$.
\end{proof}

   Suppose that $e,e'\in E^1$ with $o(e')=t(e)=x$, so that $o(e)=g_{\lambda(a)}x$ and $t(e')=g_bx$ for (unique) $a,b\in \Pi_1$. Then, by the proof of Lemma \ref{PBL},
   $\Omega(e') \subset \Omega(e)$ if and only if $b\cap\lambda(a)=(0)$.

 Equations (\ref{PB}) and (\ref{PA}) imply the following relations in $C(\bb P^n_{\bb K},\bb Z)_\Gamma$.

\begin{subequations}\label{coinv C}
\begin{eqnarray}
& \e =  & \sum_{a\in \Pi_1}[a];\label{C} \\
& [a] = & \sum_{\substack{b\in \Pi_1 \\ b\cap \lambda(a)=(0)}}[b], \quad\, a\in \Pi_1. \label{A}
\end{eqnarray}
\end{subequations}

It is easy to see that $|\Pi_1|=\frac{q^{n+1}-1}{q-1}$.
If $a\in \Pi_1$, then $\lambda(a)\in \Pi_n$ and so the number of elements
$b\in \Pi_1$ which are incident with $\lambda(a)$ is $\frac{q^n-1}{q-1}$.
Thus there exist $q^n$ elements $b\in \Pi_1$ such that $b\cap \lambda(a)=(0)$.
In other words, the right side of (\ref{A}) contains $q^n$ terms.
As a first step towards proving that $C(\bb P^n_{\bb K},\bb Z)_\Gamma$ is finite, we show that
the element $\e=[\id]$ has finite order.

\begin{lemma}\label{q2}
In the group $C(\bb P^n_{\bb K},\bb Z)_\Gamma$,  $(q^n-1)\e = 0$.
\end{lemma}

\begin{proof}
By (\ref{C}) and (\ref{A}),
  \begin{equation*}
  \e=\sum_{a\in \Pi_1}[a]=\sum_{a\in \Pi_1}\left(\sum_{\substack{b\in \Pi_1 \\ b\cap \lambda(a)=(0)}} [b]\right)
   = \sum_{b\in \Pi_1}q^n[b] = q^n\e.
\end{equation*}

\end{proof}

We can now prove Theorem \ref{finite}.
It follows from (\ref{PC}) that if $(a_1,a_2,\dots ,a_{n+1})\in \fk S$ then
 \begin{equation}\label{B}
     \sum_{i=1}^{n+1} [a_i]=\e.
 \end{equation}
Therefore, by Lemmas \ref{presentation} and \ref{fg}, there is a homomorphism
$\theta$ from $\Gamma$ onto the abelian group $C(\bb P^n_{\bb K},\bb Z)_\Gamma/\langle\e\rangle$ defined by
  $\theta(g_a)=[a]+\langle\e\rangle$, for $a\in \Pi_1$.

  The $\widetilde A_n$ group $\Gamma$ has Kazhdan's property (T) \cite[Theorems 1.6.1 and 1.7.1]{bhv}. It follows that $C(\bb P^n_{\bb K},\bb Z)_\Gamma/\langle\e\rangle$ is finite \cite[Corollary 1.3.5]{bhv}.
Therefore $C(\bb P^n_{\bb K},\bb Z)_\Gamma$ is also finite, since $\langle\e\rangle$ is finite, by Lemma \ref{q2}.
\qed

\subsection*{Distributions}\label{dist}

A {\it distribution} on $\bb P^n_{\bb K}$ is a finitely additive $\bb Z$-valued measure $\mu$ defined on the  clopen subsets of $\bb P^n_{\bb K}$ \cite[1.4]{ads}. By integration, a distribution may be regarded as a $\bb Z$-linear function on the group $C(\bb P^n_{\bb K}, \bb Z)$. Therefore a $\Gamma$-invariant distribution defines a homomorphism
$C(\bb P^n_{\bb K},\bb Z)_{\Gamma}\to \bb Z$. This homomorphism is necessarily trivial, since $C(\bb P^n_{\bb K},\bb Z)_{\Gamma}$ is finite. This proves Corollary \ref{cor}.

\end{document}